\documentclass[11pt,reqno]{amsart}
\usepackage[T1]{fontenc}
\usepackage[utf8]{inputenc}
\usepackage{lmodern,amsmath,amssymb,amsthm,mathtools}
\usepackage[a4paper,margin=27mm]{geometry}
\usepackage{microtype,enumitem,needspace}
\usepackage[hidelinks]{hyperref}
\hypersetup{pdftitle={The unique limit subvariety of the max-plus variety and finite basis properties},pdfauthor={Xiaolei Shao and Mengya Yue},pdfsubject={Varieties of additively idempotent semirings},pdfkeywords={max-plus algebra, limit variety, finite basis, ai-semiring}}
\newtheorem{theorem}{Theorem}[section]
\newtheorem{lemma}[theorem]{Lemma}
\newtheorem{proposition}[theorem]{Proposition}
\newtheorem{corollary}[theorem]{Corollary}
\theoremstyle{definition}

\newtheorem{problem}[theorem]{Problem}
\numberwithin{equation}{section}
\newcommand{\ba}{\mathbf{a}}
\newcommand{\bb}{\mathbf{b}}
\newcommand{\bu}{\mathbf{u}}
\newcommand{\bv}{\mathbf{v}}
\newcommand{\bw}{\mathbf{w}}
\newcommand{\bq}{\mathbf{q}}
\newcommand{\bp}{\mathbf{p}}
\newcommand{\br}{\mathbf{r}}
\newcommand{\bs}{\mathbf{s}}
\newcommand{\bt}{\mathbf{t}}
\newcommand{\bh}{\mathbf{h}}
\newcommand{\bN}{\mathbf{N}}
\newcommand{\V}{\mathsf{V}}
\newcommand{\Id}{\operatorname{Id}}
\newcommand{\cA}{\mathcal A}
\newcommand{\cK}{\mathcal K}
\newcommand{\cM}{\mathcal M}
\newcommand{\cD}{\mathcal D}
\newcommand{\cV}{\mathcal V}
\newcommand{\cW}{\mathcal W}
\newcommand{\cP}{\mathcal P}
\newcommand{\cT}{\mathcal T}
\newcommand{\TR}{TR_6}
\newcommand{\Xc}{X_c^+}
\newcommand{\ellD}[1]{\boldsymbol{\ell}_{#1}}
\newcommand{\len}[1]{\ell(#1)}
\newcommand{\ct}[1]{c(#1)}
\newcommand{\just}[1]{\qquad\text{by #1}}
\setlist[enumerate,1]{label=\textup{(\roman*)},leftmargin=2.3em,itemsep=2pt,topsep=4pt}
\allowdisplaybreaks[2]
\begin{document}
\title[The max-plus variety]{The unique limit subvariety of the max-plus variety and finite basis properties}
\author{Xiaolei Shao}
\author{Mengya Yue}
\begin{abstract}
We prove that the variety generated by the max-plus algebra on the nonnegative integers has a unique limit subvariety, namely the variety generated by the six-element semiring constructed by Shao, Ren and Gao. More generally, we characterize hereditary finite basedness in a finitely defined variety of commutative additively idempotent semirings containing the max-plus variety. Its hereditarily finitely based subvarieties form a finite lattice and admit equational bases involving at most ten variables. We also prove that every proper subvariety of the max-plus variety is locally finite. A finite basis theorem for finite semirings with a multiplicative identity establishes finite basedness of all finite truncations. Finally, we determine the finite basis properties of two families of subvarieties defined by power identities and of subvarieties relatively defined by arbitrary families of identities making individual additive subterms greatest elements.
\end{abstract}
\maketitle

\section{Introduction}
An \emph{additively idempotent semiring}, or an \emph{ai-semiring}, is an algebra $(S,+,\cdot)$ whose additive reduct is a commutative idempotent semigroup, whose multiplicative reduct is a semigroup, and in which multiplication distributes over addition. We consider the signature $(2,2)$, without constants. An ai-semiring is called \emph{commutative} if its multiplication is commutative.

A variety is \emph{finitely based} if it can be defined by finitely many identities, and \emph{hereditarily finitely based} if each of its subvarieties is finitely based. A \emph{limit variety} is a nonfinitely based variety all of whose proper subvarieties are finitely based. We write $\V(S)$ for the variety generated by an algebra $S$. Background on varieties and equational logic can be found in~\cite{BS}; hereditary finite basis properties of ai-semiring varieties have been studied in~\cite{RZ,YSR}.

Let $\bN=(\mathbb N_0,\max,+)$ be the max-plus algebra on the nonnegative integers, and put $\cM=\V(\bN)$. Aceto, \'Esik and Ing\'olfsd\'ottir~\cite{AEI} proved that the max-plus algebra is nonfinitely based. Ren, Jackson, Zhao and Lei~\cite[Section~4.1]{RJZL} observed that their proof also applies in the constant-free signature and established the existence of a limit subvariety of $\cM$. Shao, Ren and Gao~\cite{SRG} subsequently identified such a subvariety: they constructed the six-element commutative ai-semiring $\TR$, described its identities and its four-element subvariety lattice, and proved that $\V(\TR)$ is a limit subvariety of $\cM$.

The main result of the present paper is the uniqueness of that limit subvariety. We prove a stronger hereditary finite basis theorem. Let $\cA$ be the variety of commutative ai-semirings defined by
\begin{align}
 xy+x&\approx xy,\tag{A}\label{A}\\
 x^2+y^2+xy&\approx x^2+y^2.\tag{B}\label{B}
\end{align}
Both identities hold in $\bN$, so $\cM\leq\cA$. Let $\cK$ be the subvariety of $\cA$ defined by
\begin{equation}
 x_1x_2+x_2x_3+x_3x_4
 \approx x_1x_2+x_2x_3+x_3x_4+x_1x_4.\tag{P}\label{P}
\end{equation}
\Needspace{12\baselineskip}
\begin{theorem}\label{main}
Let $\cV\leq\cA$. The following conditions are equivalent:
\begin{enumerate}
\item $\TR\notin\cV$;
\item $\cV$ satisfies \eqref{P};
\item $\cV$ is hereditarily finitely based.
\end{enumerate}
The subvariety lattice of $\cK$ is finite, and every subvariety of $\cK$ has a finite equational basis involving at most ten variables. Moreover, $\V(\TR)$ is the unique limit subvariety of both $\cA$ and $\cM$.
\end{theorem}
The identification and nonfinite basedness of $\TR$ are results of~\cite{SRG}. Here the exclusion argument is carried out throughout $\cA$, and a representation of terms in $\cK$ yields a uniform bound on the number of variables needed to define its subvarieties. These results prove uniqueness. We also establish $\cK\leq\cM$, identifying $\cK$ as the greatest hereditarily finitely based subvariety of $\cM$.

We also investigate ordinary finite basedness and the structure of $\cM$. Every proper subvariety is locally finite and satisfies a power-stabilization identity. Every finite commutative ai-semiring satisfying \eqref{A} and possessing a multiplicative identity is finitely based. In particular, all finite truncations of $\bN$ are finitely based. We determine the finite basis properties of the varieties defined within $\cM$ by $x^k\approx x^{k+1}$ and by $x^k+z\approx x^k$. We further give a necessary and sufficient condition for finite basedness when a subvariety is relatively defined by an arbitrary family of identities $\bw+z\approx\bw$, where $\bw\in\Xc$ and $z\notin\ct{\bw}$.

Section~\ref{prelim} records the notation and the results from~\cite{SRG} needed below. Sections~\ref{exclusion}--\ref{unique} prove Theorem~\ref{main}. Sections~\ref{structure} and~\ref{finite} treat finite truncations, proper subvarieties, and finite members with identity. Sections~\ref{interval}--\ref{powers} establish the finite basis results for the stated families. Section~\ref{atoms} describes the two atoms and the inverse images of a restriction map. Section~\ref{last} concludes with a descending chain and the remaining finite basis problem.

\section{Preliminaries}\label{prelim}
Let $X$ be a countably infinite set of variables, and let $\Xc$ be the free commutative semigroup on $X$. We use bold lowercase letters for terms, and ordinary lowercase letters for variables and elements. A commutative ai-semiring term is written as
\[
 \bu=\bu_1+\cdots+\bu_m,\qquad \bu_i\in\Xc.
\]
The products $\bu_i$ are its \emph{additive subterms}. Products are identified under commutativity, and repeated additive subterms are removed by additive idempotence. Thus the free commutative ai-semiring on $X$ is $P_f(\Xc)$, with union and setwise multiplication. A substitution is a homomorphism of this free ai-semiring. Each variable is mapped to a nonempty term.

For $\bw\in\Xc$, let $\ct{\bw}$ and $\len{\bw}$ denote its content and length, respectively. An additive subterm is \emph{linear} if no variable occurs more than once. Put
\[
 \ct{\bu}=\bigcup_{i=1}^m\ct{\bu_i},\qquad
 L_k(\bu)=\{\bu_i:\len{\bu_i}=k\}.
\]
For a finite set $E\subseteq\Xc$, write $\ct E=\bigcup_{\bw\in E}\ct{\bw}$. Whenever a sum indexed by an empty set occurs as part of a displayed term, that part is omitted. For a nonempty finite set $D\subseteq X$, put
\[
 \ellD D=\sum_{x\in D}x.
\]
For an ai-semiring $S$, its additive order is given by $a\leq b$ if $a+b=b$. For terms, $\bu\preceq\bv$ abbreviates the identity $\bu+\bv\approx\bv$. We write $[\Sigma]$ for the ai-semiring variety defined by an identity set $\Sigma$ and $\Id(S)$ for the identities of $S$. Deductions between commutative ai-semiring terms are relative to the commutative ai-semiring axioms; these axioms are included when an absolute basis is asserted.

An identity $\bu\approx\bv$, where $\bu=\sum_i\bu_i$ and $\bv=\sum_j\bv_j$, is equivalent to the identities
\begin{equation}\label{one-sided}
 \bu\approx\bu+\bv_j\quad\text{for all }j,
 \qquad \bv\approx\bv+\bu_i\quad\text{for all }i.
\end{equation}
This reduction and the substitution form of equational derivations are recalled in~\cite[Section~2]{SRG}. In the presence of \eqref{A}, if $\bq$ is a nonempty factor of $\bw\in\Xc$, then
\begin{equation}\label{factor}
 \bw\approx\bw+\bq.
\end{equation}
Indeed, write $\bw=\bq\bp$ and apply \eqref{A}; if $\bq=\bw$, use additive idempotence.

We use the semiring $\TR$ with the notation and element labels of~\cite[Table~3]{SRG}. The following are the precise results from that paper used in the proof of uniqueness. The algebra $\TR$ belongs to $\cM$, and $\V(\TR)$ is a limit variety~\cite[Theorem~4.6 and Section~7]{SRG}. Further, suppose
\begin{equation}\label{failed-id}
 \bu\approx\bu+\bq,\qquad \bu_i,\bq\in\Xc,
\end{equation}
fails in $\TR$. By~\cite[Proposition~3.4]{SRG}, $\bq$ is not an additive subterm of $\bu$, every $\bu_i$ is linear of length at most two, and there are disjoint sets $A,B$ such that
\begin{equation}\label{partition}
 \ct{\bu}=A\mathbin{\dot\cup}B,
 \qquad L_2(\bu)\subseteq\{ab:a\in A,\ b\in B\}.
\end{equation}
Here variables occurring only in $L_1(\bu)$ may be assigned to either part. Moreover, $\bq$ cannot be a variable belonging to $\ct{\bu}$. These statements also cover the case in which $L_2(\bu)$ is empty.

In semiring identities we use $+$ for the semiring addition and juxtaposition for multiplication. When evaluating in $\bN$, an additive subterm is evaluated by numerical addition, and a term by the maximum of the values of its additive subterms.

\section{The exclusion identity}\label{exclusion}
The substitutions below are based on those used for the exclusion criterion inside $\V(SR_6,\TR)$ in~\cite[Proposition~6.6]{SRG}. The identities \eqref{A} and \eqref{B} allow the criterion to be proved for every subvariety of $\cA$.
\begin{theorem}\label{exclude}
For every $\cV\leq\cA$, the following are equivalent:
\[
 \TR\notin\cV,\qquad \cV\models\eqref{P},\qquad \cV\leq\cK.
\]
\end{theorem}
\begin{proof}
The last equivalence is the definition of $\cK$. In $\TR$, the substitution
$(x_1,x_2,x_3,x_4)\mapsto(2,6,5,4)$ gives
\[
 x_1x_2+x_2x_3+x_3x_4=3+3+3=3,\qquad x_1x_4=2\cdot4=1.
\]
Since $3+1=1\neq3$, identity \eqref{P} fails in $\TR$.

Conversely, assume $\TR\notin\cV$. By \eqref{one-sided}, there is an identity \eqref{failed-id} that holds in $\cV$ and fails in $\TR$. Choose $A,B$ as in \eqref{partition}, and put
\[
 \bp=x_1x_2+x_2x_3+x_3x_4,
\]
where $x_1,x_2,x_3,x_4$ are distinct variables outside $\ct{\bu}\cup\ct{\bq}$. If a substitution $\varphi$ satisfies
$\bp\approx\bp+\varphi(\bu)$ by \eqref{A}, then in $\cV$ we have
\begin{equation}\label{absorb-image}
 \begin{aligned}
 \bp&\approx\bp+\varphi(\bu)\\
 &\approx\bp+\varphi(\bu)+\varphi(\bq)\just{\eqref{failed-id}}\\
 &\approx\bp+\varphi(\bq).
 \end{aligned}
\end{equation}
We distinguish three cases.

\medskip\noindent\textbf{Case 1.} $\ct{\bq}\nsubseteq\ct{\bu}$.
Choose $t\in\ct{\bq}\setminus\ct{\bu}$ and define
\[
 \varphi(a)=
 \begin{cases}
 x_2,&a\in A,\\
 x_3,&a\in B,\\
 x_1x_4,&a=t,\\
 x_2,&a\in X\setminus(A\cup B\cup\{t\}).
 \end{cases}
\]
Each additive subterm of $\varphi(\bu)$ is $x_2$, $x_3$ or $x_2x_3$. Hence \eqref{A} gives $\bp\approx\bp+\varphi(\bu)$. Also, $x_1x_4$ is a factor of $\varphi(\bq)$. By \eqref{absorb-image} and \eqref{factor},
\[
 \bp\approx\bp+\varphi(\bq)
 \approx\bp+\varphi(\bq)+x_1x_4
 \approx\bp+x_1x_4.
\]
Thus \eqref{P} holds.

We may now assume $\ct{\bq}\subseteq\ct{\bu}$. The recalled characterization of the identities of $\TR$ gives $\len{\bq}\geq2$.

\medskip\noindent\textbf{Case 2.} At least two occurrences of variables in $\bq$ belong to the same part of \eqref{partition}.
Interchanging $A$ and $B$ if necessary, assume this part is $A$. Define $\varphi(a)=x_1$ for $a\in A$ and $\varphi(a)=x_2$ for $a\in B$; let $\varphi$ fix the other variables. Each additive subterm of $\varphi(\bu)$ is $x_1$, $x_2$ or $x_1x_2$, so $\bp\approx\bp+\varphi(\bu)$. Since $x_1^2$ is a factor of $\varphi(\bq)$, \eqref{absorb-image} and \eqref{factor} give
\begin{equation}\label{square-left}
 \bp\approx\bp+x_1^2.
\end{equation}
The substitution interchanging $x_1,x_4$ and $x_2,x_3$ fixes $\bp$ and yields
\begin{equation}\label{square-right}
 \bp\approx\bp+x_4^2.
\end{equation}
Consequently,
\[
 \begin{aligned}
 \bp&\approx\bp+x_1^2+x_4^2\just{\eqref{square-left}, \eqref{square-right}}\\
 &\approx\bp+x_1^2+x_4^2+x_1x_4\just{\eqref{B}}\\
 &\approx\bp+x_1x_4\just{\eqref{square-left}, \eqref{square-right}}.
 \end{aligned}
\]
This proves \eqref{P}.

\medskip\noindent\textbf{Case 3.} $\bq=st$, where $s\in A$ and $t\in B$.
Since $st$ is not an additive subterm of $\bu$, define
\[
 \varphi(a)=
 \begin{cases}
 x_1,&a=s,\\
 x_4,&a=t,\\
 x_3,&a\in A\setminus\{s\},\\
 x_2,&a\in B\setminus\{t\},\\
 a,&a\notin A\cup B.
 \end{cases}
\]
A member of $L_2(\bu)$ containing $s$ has image $x_1x_2$; one containing $t$ has image $x_3x_4$; one containing neither has image $x_2x_3$. None contains both $s$ and $t$. The image of an additive subterm of length one is one of $x_1,x_2,x_3,x_4$, each of which is a factor of a member of $L_2(\bp)$. Thus \eqref{A} gives $\bp\approx\bp+\varphi(\bu)$, while $\varphi(\bq)=x_1x_4$. Equation~\eqref{absorb-image} proves \eqref{P}.

Every additive subterm of length at least three belongs to Case~2, as does a square. The three cases therefore exhaust all possibilities.
\end{proof}

\section{The subvarieties of \texorpdfstring{$\cK$}{K}}\label{hfb}
Let $\Sigma_{\cK}$ consist of the commutative ai-semiring axioms and \eqref{A}, \eqref{B}, \eqref{P}. Throughout this section the displayed identities are derived from $\Sigma_{\cK}$, unless additional assumptions are specified.

\begin{lemma}\label{basic}
The variety $\cK$ satisfies
\begin{align}
 x^4&\approx x^3,\label{cube-stable}\\
 xyz&\approx x^3+y^3+z^3,\label{triple-sum}\\
 x^3y&\approx x^3+y^3,\label{cube-product}\\
 x^3+y^2&\approx x^3+y^3.\label{cube-square}
\end{align}
If $\bw\in\Xc$ and $\len{\bw}\geq3$, then
\begin{equation}\label{long-sum}
 \bw\approx\sum_{x\in\ct{\bw}}x^3.
\end{equation}
\end{lemma}
\begin{proof}
Substitute $(x^2,x,x,x^2)$ for $(x_1,x_2,x_3,x_4)$ in \eqref{P}. Then
\[
 x^3\approx x^3+x^2
 \approx x^3+x^2+x^4
 \approx x^4,
\]
where the first and last steps use \eqref{A}. This proves \eqref{cube-stable}, and hence $x^k\approx x^3$ for every $k\geq3$.

Substituting $(xy,z,y,x)$ in \eqref{P} gives
\[
 xyz\approx xyz+yz+xy
 \approx xyz+yz+xy+x^2y
 \approx xyz+x^2y.
\]
Here the first and last steps again use \eqref{A}. Replacing $(x,y,z)$ by $(x,x,y)$ in the resulting identity yields $x^2y\approx x^2y+x^3$. Therefore
\[
 xyz\approx xyz+x^2y
 \approx xyz+x^2y+x^3
 \approx xyz+x^3.
\]
Permuting the variables, we obtain
\begin{equation}\label{triple-absorb}
 xyz\approx xyz+x^3+y^3+z^3.
\end{equation}
For the reverse direction,
\[
\begin{aligned}
 x^3+y^3+z^3
 &\approx x^4+y^4+z^3+z^2\just{\eqref{cube-stable}, \eqref{A}}\\
 &\approx x^4+y^4+z^3+z^2+x^2y^2\just{\eqref{B}}\\
 &\approx x^4+y^4+z^3+z^2+x^2y^2+xyz\just{\eqref{B}}\\
 &\approx x^3+y^3+z^3+xyz\just{\eqref{B}, \eqref{cube-stable}, \eqref{A}}\\
 &\approx xyz\just{\eqref{triple-absorb}}.
\end{aligned}
\]
The two steps adding a product use the substitutions $(x,y)\mapsto(x^2,y^2)$ and $(x,y)\mapsto(xy,z)$ in \eqref{B}. This proves \eqref{triple-sum}.

Now substitute $(x^2,x,y)$ in \eqref{triple-sum}. By \eqref{cube-stable},
\[
 x^3y\approx x^6+x^3+y^3\approx x^3+y^3,
\]
which is \eqref{cube-product}. Also,
\[
\begin{aligned}
 x^3+y^2&\approx x^6+y^2
 \approx x^6+y^2+x^3y\just{\eqref{B}}\\
 &\approx x^3+y^2+x^3+y^3\just{\eqref{cube-stable}, \eqref{cube-product}}\\
 &\approx x^3+y^3\just{\eqref{A}}.
\end{aligned}
\]
Finally, prove \eqref{long-sum} by induction on length. The case of length three is \eqref{triple-sum}. If $\bw=\bv t$, where $t\in X$ and $\len{\bv}\geq3$, then
\[
 \bw\approx\left(\sum_{x\in\ct{\bv}}x^3\right)t
 \approx\sum_{x\in\ct{\bv}}(x^3+t^3)
 \approx\sum_{x\in\ct{\bw}}x^3.
\]
\end{proof}

\begin{lemma}\label{sum-powers}
For a nonempty finite set $D\subseteq X$ and $r\in\{2,3\}$,
\begin{equation}\label{powers-of-sums}
 \ellD D^{\,r}\approx\sum_{x\in D}x^r.
\end{equation}
\end{lemma}
\begin{proof}
By \eqref{B} and \eqref{triple-sum}, respectively,
\[
 (x+y)^2\approx x^2+xy+y^2\approx x^2+y^2,
\]
\[
 (x+y)^3\approx x^3+x^2y+xy^2+y^3\approx x^3+y^3.
\]
Induction on $|D|$ proves the assertion.
\end{proof}

\begin{corollary}\label{K-lf}
The variety $\cK$ is locally finite.
\end{corollary}
\begin{proof}
For a finite variable set $Z$, equation~\eqref{long-sum} expresses every $Z$-term as a nonempty sum of members of the finite set
\[
 \{x,x^2,x^3:x\in Z\}\cup\{xy:x,y\in Z,\ x\neq y\}.
\]
Additive idempotence leaves only finitely many such sums.
\end{proof}

\Needspace{10\baselineskip}
\begin{lemma}\label{quadratic}
Let $\bh$ be a nonempty sum of additive subterms of length two, and let $Y=\ct{\bh}$. There are pairwise disjoint sets $C,U_1,V_1,\ldots,U_m,V_m$, where $m\geq0$ and each $U_i,V_i$ is nonempty, such that
\begin{equation}\label{quadratic-rep}
 Y=C\mathbin{\dot\cup}\bigcup_{i=1}^m(U_i\mathbin{\dot\cup}V_i),
 \qquad \bh\approx\sum_{x\in C}x^2+\sum_{i=1}^m\ellD{U_i}\ellD{V_i}.
\end{equation}
If $x^2$ occurs in $\bh$, then $x\in C$. Each additive subterm $xy$ of $\bh$ has either both factors in $C$, or one in $U_i$ and the other in $V_i$ for some $i$.
\end{lemma}
\begin{proof}
Write $Y=\{a_1,\ldots,a_n\}$, with the variables $a_i$ pairwise distinct, and put
\[
 \mathcal S=\{a_i a_j:1\leq i\leq j\leq n,
       \ \Sigma_{\cK}\vdash\bh\approx\bh+a_i a_j\}.
\]
This is a finite set containing every additive subterm of $\bh$. By adding its members successively,
\begin{equation}\label{s-sum}
 \bh\approx\bh+\sum_{\bw\in\mathcal S}\bw
 \approx\sum_{\bw\in\mathcal S}\bw.
\end{equation}
The second step uses additive idempotence. If $ab,bc,cd\in\mathcal S$, then
\[
 \bh\approx\bh+ab+bc+cd
 \approx\bh+ab+bc+cd+ad\just{\eqref{P}}
 \approx\bh+ad.
\]
The last step removes $ab,bc,cd$ by the identities defining their membership in $\mathcal S$. Thus
\begin{equation}\label{three-products}
 ab,bc,cd\in\mathcal S\quad\Longrightarrow\quad ad\in\mathcal S.
\end{equation}
The variables in this implication need not be distinct. Taking $(a,b,c,d)=(b,a,a,b)$ gives
\begin{equation}\label{square-propagation}
 a^2,ab\in\mathcal S\quad\Longrightarrow\quad b^2\in\mathcal S.
\end{equation}
Put $C=\{a\in Y:a^2\in\mathcal S\}$ and $D=Y\setminus C$. By \eqref{square-propagation}, no member of $\mathcal S$ has one factor in $C$ and the other in $D$. The sum of the members with content in $C$ is equivalent to $\sum_{a\in C}a^2$, since \eqref{B} removes each additional product.

Suppose $D\neq\varnothing$. Every variable of $Y$ occurs in $\mathcal S$, so there is $ab\in\mathcal S$ with $a,b\in D$. Necessarily $a\neq b$. Define
\[
 U=\{x\in D:xb\in\mathcal S\},\qquad
 V=\{y\in D:ay\in\mathcal S\}.
\]
These sets are nonempty because $a\in U$ and $b\in V$. If $z\in U\cap V$, then $zb,ba,az\in\mathcal S$, so \eqref{three-products} gives $z^2\in\mathcal S$, contrary to $z\in D$. Hence $U\cap V=\varnothing$. For $x\in U$ and $y\in V$, the products $xb,ba,ay$ belong to $\mathcal S$, so
\begin{equation}\label{complete-products}
 xy\in\mathcal S\qquad(x\in U,\ y\in V).
\end{equation}
Conversely, let $x\in U$ and $xz\in\mathcal S$. Then $z\notin C$ by \eqref{square-propagation}; from $ab,bx,xz\in\mathcal S$, equation~\eqref{three-products} gives $az\in\mathcal S$. Thus $z\in V$. Similarly, if $y\in V$ and $yz\in\mathcal S$, then $ba,ay,yz\in\mathcal S$ give $bz\in\mathcal S$, so $z\in U$. Consequently,
\[
 \{\bw\in\mathcal S:\ct{\bw}\cap(U\cup V)\neq\varnothing\}
 =\{xy:x\in U,\ y\in V\}.
\]
Remove $U\cup V$ from $D$ and repeat the construction. The displayed equality ensures that each remaining variable still occurs with another remaining variable. Since $D$ is finite, the process gives disjoint pairs $U_i,V_i$ exhausting $D$. Their products have sum $\sum_i\ellD{U_i}\ellD{V_i}$, by distributivity. Equation~\eqref{s-sum} now proves \eqref{quadratic-rep}. The final assertions follow because all original additive subterms of $\bh$ belong to $\mathcal S$.
\end{proof}

\begin{proposition}\label{representation}
Every term is equivalent in $\cK$ to a term
\begin{equation}\label{normal}
 \bu=\sum_{x\in L}x+\sum_{x\in C}x^r
      +\sum_{i=1}^m\ellD{U_i}\ellD{V_i},
\end{equation}
where $r\in\{2,3\}$, the sets $L,C,U_1,V_1,\ldots,U_m,V_m$ are pairwise disjoint, their union is nonempty, and each $U_i,V_i$ is nonempty.
\end{proposition}
\begin{proof}
By \eqref{long-sum}, every term is equivalent to a sum of additive subterms of the forms $x,x^2,x^3,xy$ with $x\neq y$. If an additive subterm $x$ of length one also occurs as a factor of a longer additive subterm, remove $x$ using \eqref{A}. Let $L$ be the set of remaining variables of length one, and let $Q,T$ be the sets of variables whose squares and cubes occur. Let $E$ be the set of the remaining products of distinct variables. Adding the squares of the variables in $T$, by \eqref{A}, gives
\[
 \sum_{x\in L}x+\bh+\sum_{x\in T}x^3,
 \qquad
 \bh=\sum_{x\in Q\cup T}x^2+\sum_{\bw\in E}\bw.
\]
If the two indexing sets in $\bh$ are empty, the original term is a sum of variables. Otherwise apply Lemma~\ref{quadratic} to $\bh$. We obtain
\[
 \sum_{x\in L}x+\sum_{x\in C}x^2+\sum_{x\in T}x^3
       +\sum_{i=1}^m\ellD{U_i}\ellD{V_i},
 \qquad T\subseteq C.
\]
If $T$ is empty, take $r=2$. If $T$ is nonempty, choose $z\in T$ and use \eqref{cube-square} for each $x\in C\setminus\{z\}$:
\[
 \begin{aligned}
 \sum_{x\in C}x^2+\sum_{x\in T}x^3
 &\approx z^3+\sum_{x\in C\setminus\{z\}}x^2
             +\sum_{x\in T\setminus\{z\}}x^3\\
 &\approx z^3+\sum_{x\in C\setminus\{z\}}x^3
             +\sum_{x\in T\setminus\{z\}}x^3
 \approx\sum_{x\in C}x^3.
 \end{aligned}
\]
This yields \eqref{normal} with $r=3$.
\end{proof}

\begin{corollary}\label{original-support}
Let $\bv=\bv_1+\cdots+\bv_s$ be a term and choose its representation \eqref{normal} by the preceding construction. Then:
\begin{enumerate}
\item a variable in $L$ occurs in $\bv$ only as an additive subterm of length one;
\item an additive subterm of $\bv$ involving a variable in $U_i\cup V_i$ is either a single variable or a product $ab$ with $a\in U_i$ and $b\in V_i$;
\item if some $\bv_j$ has length at least three, the representation has $r=3$.
\end{enumerate}
\end{corollary}
\begin{proof}
The definition of $L$ proves (i). A variable occurring in a square or in an additive subterm of length at least three has its square in $\bh$, and therefore belongs to $C$. Lemma~\ref{quadratic} gives (ii). A subterm of length at least three makes $T$ nonempty, which gives (iii).
\end{proof}

By \eqref{one-sided} and Proposition~\ref{representation}, every identity is equivalent within $\cK$ to finitely many identities $\bu\approx\bu+\bq$, where $\bu$ has form \eqref{normal} and $\bq$ is one of $x,x^2,x^3,xy$, with $x\neq y$ in the last case.
\begin{lemma}\label{ten}
Each such identity is equivalent within $\cK$ to one identity involving at most ten variables.
\end{lemma}
\begin{proof}
Put $D=\ct{\bq}$. Renumber the pairs $U_i,V_i$ so that
\[
 (U_i\cup V_i)\cap D\neq\varnothing\quad(1\leq i\leq k),\qquad
 (U_i\cup V_i)\cap D=\varnothing\quad(k<i\leq m).
\]
Since the sets are pairwise disjoint and $|D|\leq2$, we have $k\leq2$. Choose distinct new variables $a,b,a_i,b_i$ for $1\leq i\leq k$, and $z_1,z_2$, outside $\ct{\bu}\cup D$. Define a substitution fixing $D$ by
\[
 \varphi(x)=
 \begin{cases}
 a,&x\in L\setminus D,\\
 b,&x\in C\setminus D,\\
 a_i,&x\in U_i\setminus D,\quad 1\leq i\leq k,\\
 b_i,&x\in V_i\setminus D,\quad 1\leq i\leq k,\\
 z_1,&x\in U_i,\quad k<i\leq m,\\
 z_2,&x\in V_i,\quad k<i\leq m.
 \end{cases}
\]
Other variables are fixed. Put $\bu_0=\varphi(\bu)$. Every pair with $i>k$ has image $z_1z_2$, and $\varphi(\bq)=\bq$. Therefore $\bu\approx\bu+\bq$ implies
\begin{equation}\label{reduced-id}
 \bu_0\approx\bu_0+\bq,
\end{equation}
which involves at most $|D|+2+2k+2\leq10$ variables.

To prove the converse, define $\psi$ to fix $D$ and set
\[
 \psi(a)=\ellD{L\setminus D},\quad
 \psi(b)=\ellD{C\setminus D},\quad
 \psi(a_i)=\ellD{U_i\setminus D},\quad
 \psi(b_i)=\ellD{V_i\setminus D}\quad(1\leq i\leq k).
\]
Each prescription is used only when the corresponding variable occurs in $\bu_0$, so its right-hand side is nonempty. Distributivity restores the first $k$ product parts. For the power part, Lemma~\ref{sum-powers} gives
\[
 \psi\varphi\left(\sum_{x\in C}x^r\right)
 \approx\sum_{x\in C\cap D}x^r+
       \left(\sum_{x\in C\setminus D}x\right)^r
 \approx\sum_{x\in C}x^r.
\]
The part of length one is restored as well.

If $k=m$, then $\psi(\bu_0)\approx\bu$, and applying $\psi$ to \eqref{reduced-id} proves the assertion. If $k<m$, also set
\[
 \psi(z_1)=\ellD{U_{k+1}},\qquad
 \psi(z_2)=\ellD{V_{k+1}}.
\]
Then $\psi(\bu_0)\approx\bh$, where
\[
 \bh=\sum_{x\in L}x+\sum_{x\in C}x^r
             +\sum_{i=1}^{k+1}\ellD{U_i}\ellD{V_i}.
\]
Each additive subterm of $\bh$, after distributing, occurs in $\bu$, so $\bu+\bh\approx\bu$. Since $\psi(\bq)=\bq$, equation~\eqref{reduced-id} yields
\[
 \begin{aligned}
 \bu&\approx\bu+\bh
 \approx\bu+\psi(\bu_0)\\
 &\approx\bu+\psi(\bu_0+\bq)\just{\eqref{reduced-id}}\\
 &\approx\bu+\bh+\bq
 \approx\bu+\bq.
 \end{aligned}
\]
Thus \eqref{reduced-id} is equivalent to the original identity within $\cK$.
\end{proof}

\begin{theorem}\label{K-finite-lattice}
The subvariety lattice of $\cK$ is finite. Every subvariety of $\cK$ has a finite equational basis involving at most ten variables.
\end{theorem}
\begin{proof}
Fix $Z=\{z_1,\ldots,z_{10}\}$. By Corollary~\ref{K-lf}, the relatively free algebra $F_{\cK}(Z)$ is finite. Choose a finite set $R$ of terms representing its elements, and put
\[
 \Theta=\{\br\approx\bs:\br,\bs\in R\}.
\]
Let $\cV=\cK\cap[\Sigma]$ be a subvariety of $\cK$. By \eqref{one-sided}, Proposition~\ref{representation}, and Lemma~\ref{ten}, each identity $\varepsilon\in\Sigma$ is equivalent in $\cK$ to finitely many identities on at most ten variables. Rename their variables into $Z$ and replace each term by its representative in $R$. This gives a subset $\Theta_\varepsilon\subseteq\Theta$ equivalent to $\varepsilon$ within $\cK$.

The set $\Delta=\bigcup_{\varepsilon\in\Sigma}\Theta_\varepsilon$ is a subset of the finite set $\Theta$, and $\cV=\cK\cap[\Delta]$. Hence only finitely many subvarieties of $\cK$ occur. Moreover, $\Sigma_{\cK}\cup\Delta$ is a finite basis for $\cV$. The identities in $\Sigma_{\cK}$ use at most four variables, and those in $\Delta$ use at most ten.
\end{proof}

\begin{corollary}\label{K-hfb}
The variety $\cK$ is hereditarily finitely based. For every identity set $\Sigma$, there is a finite subset $\Sigma_0\subseteq\Sigma$ such that
\[
 \cK\cap[\Sigma]=\cK\cap[\Sigma_0].
\]
\end{corollary}
\begin{proof}
The first assertion follows from Theorem~\ref{K-finite-lattice}. In its proof, for each $\delta\in\Delta$ choose $\varepsilon_\delta\in\Sigma$ with $\delta\in\Theta_{\varepsilon_\delta}$. Then the finite set $\Sigma_0=\{\varepsilon_\delta:\delta\in\Delta\}$ is equivalent to $\Delta$, and hence to $\Sigma$, relative to $\cK$.
\end{proof}

\section{The unique limit subvariety}\label{unique}
\begin{proof}[Proof of Theorem~\ref{main}]
By Theorem~\ref{exclude}, conditions (i) and (ii) are equivalent, and imply $\cV\leq\cK$. Theorem~\ref{K-finite-lattice} then gives (iii). Conversely, a hereditarily finitely based variety cannot contain the nonfinitely based variety $\V(\TR)$, so (iii) implies (i).

Let $\mathcal L\leq\cA$ be a limit variety. If $\TR\notin\mathcal L$, the equivalence just proved makes $\mathcal L$ finitely based, a contradiction. Thus $\V(\TR)\leq\mathcal L$. Since $\V(\TR)$ is nonfinitely based, it cannot be a proper subvariety of $\mathcal L$. Therefore $\mathcal L=\V(\TR)$. The cited results of~\cite{SRG} give that $\V(\TR)$ is a limit variety and $\TR\in\cM\leq\cA$. This proves uniqueness in both ambient varieties. The finite-lattice assertion and the variable bound are Theorem~\ref{K-finite-lattice}.
\end{proof}

\begin{proposition}\label{K-in-M}
The variety $\cK$ is a subvariety of $\cM$.
\end{proposition}
\begin{proof}
Let $\bv\approx\bv+\bq$ hold in $\bN$, where $\bv=\bv_1+\cdots+\bv_s$ and $\bv_j,\bq\in\Xc$. Assigning $1$ to a variable of $\bq$ outside $\ct{\bv}$ and $0$ to the other variables proves
\begin{equation}\label{content-N}
 \ct{\bq}\subseteq\ct{\bv}.
\end{equation}
Put $d=\max_j\len{\bv_j}$, and let $\bu$ be the representation \eqref{normal} of $\bv$ obtained in Proposition~\ref{representation}. We prove $\bv\approx\bv+\bq$ in $\cK$.

If $\len{\bq}=1$, use \eqref{content-N} and \eqref{factor}. If $\bq=x^2$, evaluate $x$ at $1$ and all other variables at $0$ in $\bN$. Some $\bv_j$ must contain $x$ at least twice, so \eqref{factor} again applies.

Suppose $\bq=xy$, with $x\neq y$. If $x,y\in C$, the power part of $\bu$ absorbs $x^2+y^2$, and then absorbs $xy$ by \eqref{B}. If $x\in U_i$ and $y\in V_i$, or conversely, $xy$ is an additive subterm of $\bu$ after distributing.

In any other case, at least one of $x,y$ lies outside $C$; denote it by $x$. Let $D$ be the part $U_i$ or $V_i$ containing $x$, or let $D=\{x\}$ if $x\in L$. If $y\in D$, assign $1$ to every variable in $D$ and $0$ elsewhere. By Corollary~\ref{original-support}, each original $\bv_j$ contains at most one occurrence from $D$. Thus $\bv\leq1<2=xy$, a contradiction.

If $y\notin D$, it does not lie in the opposite part of the same pair. Choose an integer $N>d$ and evaluate
\[
 t\mapsto\begin{cases}
 N,&t\in D,\\1,&t=y,\\0,&\text{otherwise}.
 \end{cases}
\]
An original additive subterm involving $D$ is a single variable from $D$, or a product whose other variable has value $0$. It has value at most $N$. Every other original additive subterm has value at most $d<N$. Hence $\bv\leq N<N+1=xy$, again a contradiction. This settles length two.

Finally, let $\len{\bq}\geq3$. Evaluating all variables at $1$ gives $d\geq\len{\bq}\geq3$, so Corollary~\ref{original-support} gives $r=3$. Suppose $x\in\ct{\bq}\setminus C$, and define $D$ as above. Assign $N>d$ to the variables in $D$ and $1$ to all other variables. The original additive subterms involving $D$ have value at most $N+1$; all others have value at most $d$. But
\[
 \bq\geq N+\len{\bq}-1\geq N+2,
\]
a contradiction. Thus $\ct{\bq}\subseteq C$. Equation~\eqref{long-sum} gives $\bq\approx\sum_{x\in\ct{\bq}}x^3$, which is absorbed by the power part of $\bu$.

Every identity of $\bN$ therefore holds in $\cK$, by \eqref{one-sided}. Hence $\cK\leq\V(\bN)$.
\end{proof}
\begin{corollary}\label{greatest-hfb}
The variety $\cK$ is the greatest hereditarily finitely based subvariety of both $\cA$ and $\cM$. Their hereditarily finitely based subvarieties are the same finite family.
\end{corollary}
\begin{proof}
Theorem~\ref{main} identifies these varieties with the subvarieties of $\cK$, and Proposition~\ref{K-in-M} places them all in $\cM$.
\end{proof}

\section{Finite truncations and proper subvarieties}\label{structure}
For $k\geq0$, put $I_k=\{k,k+1,k+2,\ldots\}$. The Rees quotient $A_k=\bN/I_k$ has universe $\{0,1,\ldots,k\}$ and operations
\begin{equation}\label{truncations}
 a+b=\max\{a,b\},\qquad a\cdot b=\min\{a+b,k\},
\end{equation}
where the sum inside the minimum is numerical addition. The algebra $A_0$ is trivial, and numerical $0$ is a multiplicative identity in each $A_k$. The algebra $A_2$ is isomorphic to the semiring $S_{53}$ used in~\cite[Theorem~1.2]{SRG}.

\begin{proposition}\label{congruences}
Every congruence of $\bN$ other than equality is the Rees congruence of a unique $I_k$. Moreover, $\bN$ is residually finite, and
\begin{equation}\label{Ak-chain}
 \V(A_0)<\V(A_1)<\V(A_2)<\cdots,
 \qquad \bigvee_{k\geq1}\V(A_k)=\cM.
\end{equation}
\end{proposition}
\begin{proof}
Each $I_k$ is an additive order filter and a multiplicative ideal, so its Rees relation is a congruence. Conversely, let $\rho$ be a nonidentity congruence, and choose the least integer $k$ congruent to a larger integer $l$. Compatibility with maximum gives
\[
 k\mathrel\rho k+1\mathrel\rho\cdots\mathrel\rho l.
\]
In particular, $k\mathrel\rho k+1$. Compatibility with numerical addition then gives $k+j\mathrel\rho k+j+1$ for every $j\geq0$. Thus the integers at least $k$ form one class; the choice of $k$ makes each smaller integer a singleton class.

Distinct $a,b\in\mathbb N_0$ have distinct images in $A_k$ whenever $k>\max\{a,b\}$. Hence $\bN$ embeds in $\prod_{k\geq1}A_k$, proving residual finiteness and the join assertion. If $k\leq l$, the map $a\mapsto\min\{a,k\}$ is an epimorphism $A_l\to A_k$. For $k\geq1$, the identity $x^k\approx x^{k+1}$ holds in $A_k$ but fails in $A_{k+1}$ at $x=1$. This proves the strict inclusions.
\end{proof}

\begin{proposition}\label{embeddings}
Let $2\leq k\leq l$. There is an embedding $A_k\hookrightarrow A_l$ if and only if some integer $r\geq1$ satisfies
\begin{equation}\label{embedding-condition}
 (k-1)r<l\leq kr.
\end{equation}
In particular, $A_k$ does not embed in $A_{k+1}$ for $k\geq3$.
\end{proposition}
\begin{proof}
The multiplicatively idempotent elements of $A_l$ are $0$ and $l$. If an embedding $f$ sent $0$ to $l$, the equalities $0\cdot a=a$ would force $f(a)=l$ for every $a$. Thus $f(0)=0$. Put $r=f(1)>0$. Since each positive integer $j\leq k$ is a power of numerical $1$,
\[
 f(j)=\min\{jr,l\}\qquad(1\leq j\leq k).
\]
Also, $k$ is multiplicatively idempotent and has nonzero image, so $f(k)=l$. Injectivity is therefore equivalent to \eqref{embedding-condition}.

Conversely, under \eqref{embedding-condition}, the displayed formula and $f(0)=0$ define a strictly increasing map, which preserves maximum. For $a,b\in A_k$, numerical calculation gives
\[
 \min\{\min\{a+b,k\}r,l\}
 =\min\{(a+b)r,l\}
 =\min\{f(a)+f(b),l\}.
\]
Thus it also preserves multiplication. Finally, if $l=k+1$, the inequality $l\leq kr$ forces $r\geq2$, whereas $(k-1)r\geq2k-2\geq k+1$ for $k\geq3$.
\end{proof}

\begin{lemma}\label{positive-values}
If $\bu\approx\bu+\bq$ fails in $\bN$, then, for every integer $h\geq0$, it fails under an evaluation assigning every variable an integer greater than $h$.
\end{lemma}
\begin{proof}
Let $\mathbf a=(a_1,\ldots,a_n)$ witness failure, so that $\bq(\mathbf a)>\bu_i(\mathbf a)$ for every additive subterm $\bu_i$ of $\bu$. Set $b_j=Na_j+h+1$. Then
\[
 \bq(\mathbf b)-\bu_i(\mathbf b)
 =N\bigl(\bq(\mathbf a)-\bu_i(\mathbf a)\bigr)
 +(h+1)\bigl(\len{\bq}-\len{\bu_i}\bigr).
\]
For a sufficiently large integer $N$, all these differences are positive.
\end{proof}
\begin{corollary}\label{tails}
For each $h\geq0$, the subsemiring $\{h+1,h+2,\ldots\}$ of $\bN$ generates $\cM$.
\end{corollary}
\begin{proof}
By \eqref{one-sided} and Lemma~\ref{positive-values}, every identity failing in $\bN$ also fails in this subsemiring.
\end{proof}

We recall the exponent-vector description of max-plus identities from~\cite[Section~3]{AEI}. For $\bw\in\Xc$ over $x_1,\ldots,x_n$, write
$\alpha(\bw)\in\mathbb N_0^n\setminus\{\mathbf0\}$ for its vector of multiplicities. If $\bu=\bu_1+\cdots+\bu_m$, then
\begin{equation}\label{convex}
 \bN\models\bu\approx\bu+\bq
 \quad\Longleftrightarrow\quad
 \alpha(\bq)\leq\sum_{i=1}^m\lambda_i\alpha(\bu_i)
 \text{ for some }\lambda_i\geq0,\quad\sum_{i=1}^m\lambda_i=1,
\end{equation}
where the inequality is coordinatewise. For a finite nonempty set $U\subseteq\mathbb N_0^n\setminus\{\mathbf0\}$, define
\[
 C(U)=\left\{\mathbf b\in\mathbb N_0^n\setminus\{\mathbf0\}:
 \mathbf b\leq\sum_{\mathbf a\in U}\lambda_{\mathbf a}\mathbf a,
 \quad\lambda_{\mathbf a}\geq0,\quad\sum_{\mathbf a\in U}\lambda_{\mathbf a}=1\right\}.
\]
In the constant-free signature, the corresponding relatively free algebra has the sets $C(U)$ as its elements, with
\[
 C(U)\oplus C(V)=C(U\cup V),\qquad
 C(U)\odot C(V)=C(U+V),
\]
where $U+V=\{\mathbf a+\mathbf b:\mathbf a\in U,\mathbf b\in V\}$. Its generators are $C(\{\mathbf e_i\})$, where $\mathbf e_i$ are the coordinate vectors. Indeed, \eqref{convex} identifies two terms precisely when their exponent sets have the same $C$-closure; term addition and multiplication correspond to union and the set of pairwise sums. We use \eqref{convex} in Section~\ref{powers}.

\begin{theorem}\label{proper-lf}
Every proper subvariety $\cV<\cM$ satisfies $x^s\approx x^{s+1}$ for some $s\geq1$. In particular, every proper subvariety of $\cM$ is locally finite.
\end{theorem}
\begin{proof}
Choose an identity $\bu\approx\bu+\bq$ that holds in $\cV$ and fails in $\bN$. By Lemma~\ref{positive-values}, choose a failure evaluation $x_j\mapsto a_j$ with all $a_j>0$. Put
\[
 s=\max_i\bu_i(a_1,\ldots,a_n),\qquad
 t=\bq(a_1,\ldots,a_n).
\]
Then $1\leq s<t$. Substitute $x_j\mapsto x^{a_j}$ in the identity. By \eqref{A}, the image of $\bu$ is equivalent to $x^s$, and the resulting identity gives $x^s\approx x^t$. Since \eqref{A} also gives $x^s\preceq x^{s+1}\preceq x^t$,
\[
 x^{s+1}\approx x^s+x^{s+1}
 \approx x^t+x^{s+1}\approx x^t\approx x^s.
\]
For terms over $n$ variables, each positive exponent can now be reduced to at most $s$. There are at most $(s+1)^n-1$ resulting additive subterms and only finitely many nonempty sums of them. The relatively free algebra of each finite rank is therefore finite.
\end{proof}
\begin{corollary}\label{M-not-lf}
The variety $\cM$ is not locally finite and is not generated by a finite algebra.
\end{corollary}
\begin{proof}
The subsemiring of $\bN$ generated by numerical $1$ contains every positive integer. A variety generated by a finite algebra is locally finite, since its relatively free algebra of finite rank is a finite algebra of term functions.
\end{proof}

\section{Finite semirings with a multiplicative identity}\label{finite}
\begin{theorem}\label{unit-fb}
Let $S$ be a finite commutative ai-semiring satisfying \eqref{A} and possessing a multiplicative identity. Then $S$ is finitely based in the signature $(2,2)$. If $|S|=d+1>1$, a basis can be chosen whose identities, apart from the commutative ai-semiring axioms, involve at most $d+1$ variables.
\end{theorem}
\begin{proof}
The one-element case is immediate. Suppose $|S|=d+1>1$, and let $e$ be its multiplicative identity. For $a\in S$, identity~\eqref{A} gives $e+a=e+ea=ea=a$. Thus $e$ is the least element of the additive order.

We first prove that $S$ satisfies
\begin{equation}\label{length-reduction}
 x_1\cdots x_{d+1}\approx
 \sum_{i=1}^{d+1}x_1\cdots x_{i-1}x_{i+1}\cdots x_{d+1}.
\end{equation}
For arbitrary $a_1,\ldots,a_{d+1}\in S$, consider
\[
 e\leq a_1\leq a_1a_2\leq\cdots\leq a_1\cdots a_{d+1}.
\]
There are $d+2$ entries and only $d+1$ elements of $S$, so two adjacent entries coincide. Multiplication by the remaining factors shows that deleting the corresponding factor does not change the full product. Each product obtained by deleting a factor is at most the full product, by \eqref{A}. Their sum is therefore the full product, proving \eqref{length-reduction}.

Repeated use of \eqref{length-reduction} reduces every term to a sum of additive subterms of length at most $d$. Fix $d$ variables, and let $R_d$ consist of all such terms on these variables, modulo commutativity and additive idempotence. This set is finite. Let $\Gamma$ be the set of all identities between members of $R_d$ that hold in $S$.

We prove that the commutative ai-semiring axioms, \eqref{A}, \eqref{length-reduction}, and $\Gamma$ form a basis for $S$. By the length reduction and \eqref{one-sided}, it suffices to derive $\bu\approx\bu+\bq$ when every additive subterm involved has length at most $d$.

First, $\ct{\bq}\subseteq\ct{\bu}$. Otherwise assign an element $a\neq e$ to a variable of $\bq$ absent from $\bu$, and assign $e$ to every other variable. The value of $\bu$ is $e$, whereas that of $\bq$ is $a^r\geq a>e$, a contradiction.

Delete from each additive subterm of $\bu$ the variables outside $\ct{\bq}$, and omit any additive subterm all of whose variables are deleted. Denote the resulting nonempty term by $\bv$. Evaluating the deleted variables at $e$ in $S$ shows that $S\models\bv\approx\bv+\bq$. This identity involves at most $\len{\bq}\leq d$ variables, and its additive subterms have length at most $d$. After renaming variables, it belongs to $\Gamma$. Each additive subterm of $\bv$ is a nonempty factor of an additive subterm of $\bu$. Hence
\[
 \bu\approx\bu+\bv\approx\bu+\bv+\bq\approx\bu+\bq,
\]
by \eqref{factor} and $\Gamma$. This proves completeness and the stated variable bound.
\end{proof}
\begin{corollary}\label{Ak-fb}
Every finite truncation $A_k$ is finitely based.
\end{corollary}
\begin{proof}
Apply Theorem~\ref{unit-fb} to the multiplicative identity $0$ of $A_k$; the case $k=0$ is trivial.
\end{proof}
The case $k=2$ has the following short basis, established in~\cite[Proposition~7]{Zhao} and used in~\cite[Theorem~1.2]{SRG}: relative to the commutative ai-semiring axioms, $\V(A_2)$ is defined by \eqref{A} and
\begin{align}
 xy+y^2&\approx x+y^2,\label{A2-basis1}\\
 xyz&\approx xy+yz+xz.\label{A2-basis2}
\end{align}
We shall use this basis in Section~\ref{last}.

For a commutative ai-semiring $S$ satisfying \eqref{A}, adjoin a new element $e$ and retain the operations on $S$, with
\[
 e+a=a+e=a,\qquad ea=ae=a\quad(a\in S),\qquad e+e=ee=e.
\]
Denote the resulting algebra by $S^1$.
\begin{proposition}\label{adjoin}
The algebra $S^1$ is a commutative ai-semiring satisfying \eqref{A}. If $S\in\cM$, then $S^1\in\cM$.
\end{proposition}
\begin{proof}
The extended operations are associative and commutative. For distributivity, the new mixed case is
\[
 (e+a)b=ab=b+ab=eb+ab\qquad(a,b\in S),
\]
where the middle equality is \eqref{A}. Cases with multiplier $e$ follow from its definition as an identity. Identity~\eqref{A} is also preserved.

Let $S\in\cM$ and let $\bu\approx\bu+\bq$ hold in $\bN$. By \eqref{content-N}, $\ct{\bq}\subseteq\ct{\bu}$. Fix an evaluation in $S^1$ and let $Y$ be the set of variables assigned values in $S$. If $\bq$ has value $e$, the identity follows because $e$ is least. Otherwise $\ct{\bq}\cap Y\neq\varnothing$.

Delete the variables outside $Y$ from all additive subterms of $\bu$ and from $\bq$, omitting products that become empty. Let the resulting terms be $\bv$ and $\br$. Both are nonempty. Evaluation of the deleted variables at numerical $0$ in $\bN$ proves $\bN\models\bv\approx\bv+\br$, so the identity holds in $S$. Under the fixed evaluation, $\bu$ and $\bq$ in $S^1$ have the values of $\bv$ and $\br$ in $S$. Hence $\bu\approx\bu+\bq$ holds in $S^1$. Now apply \eqref{one-sided}.
\end{proof}
\begin{corollary}\label{finite-embedding}
Every finite member of $\cM$ embeds in a finite, finitely based member of $\cM$.
\end{corollary}
\begin{proof}
For finite $S\in\cM$, use $S\subseteq S^1$, Proposition~\ref{adjoin}, and Theorem~\ref{unit-fb}.
\end{proof}

\section{An interval of nonfinitely based varieties}\label{interval}
For $k\geq2$, define
\begin{equation}\label{W-D}
 \cW_k=\cM\cap[x^k+z\approx x^k],\qquad
 \cD=\cM\cap[x^3+z\approx x^3,\ x^2y^2+z\approx x^2y^2].
\end{equation}
An identity $\bw+z\approx\bw$ with $z\notin\ct{\bw}$ means that every value of $\bw$ is an additive greatest element. In a semiring satisfying \eqref{A}, such an element $b$ is multiplicatively absorbing: $b\leq ba\leq b$ gives $ba=b$.

For $n\geq3$, put
\begin{equation}\label{aei-terms}
 \bp_n=x_1\cdots x_n,\qquad
 \bv_n=\sum_{i=1}^n x_i^2
           \prod_{\substack{1\leq j\leq n\\j\notin\{i,i+1\}}}x_j,
\end{equation}
where $i+1$ is interpreted modulo $n$. Write $\varepsilon_n$ for $\bv_n\approx\bv_n+\bp_n$. These identities hold in $\bN$: if $x_i$ has value $a_i$ and $s=\sum_i a_i$, then $\bp_n$ has value $s$, and the $i$th additive subterm of $\bv_n$ has value $s+a_i-a_{i+1}$. At least one of the latter values is at least $s$.

The proof of~\cite[Theorems~4.1 and~4.2]{AEI} gives the following property:
\begin{equation}\label{AEI-independent}
 \text{no finite subset of }\Id(\bN)\text{ implies all }\varepsilon_n,
 \qquad n\geq3.
\end{equation}
Its applicability in the constant-free signature is recorded in~\cite[Section~4.1]{RJZL}.

\begin{lemma}\label{no-first-use}
Let $\Sigma\subseteq\Id(\bN)$. If
\[
 \Sigma\cup\{x^3+z\approx x^3,\ x^2y^2+z\approx x^2y^2\}
 \vdash\varepsilon_n,
\]
then $\Sigma\vdash\varepsilon_n$.
\end{lemma}
\begin{proof}
Suppose $\bN\models\bt\approx\bv_n$. No additive subterm of $\bt$ contains a variable three times: assigning $1$ to that variable and $0$ to all others would give $\bt\geq3$ and $\bv_n\leq2$. Nor can an additive subterm contain two distinct variables each at least twice: assigning $1$ to those two variables and $0$ elsewhere would give $\bt\geq4$ and $\bv_n\leq3$.

Consider a derivation starting from $\bv_n$. Before the first use of either added identity, the current term is equal to $\bv_n$ in $\bN$, and therefore has the two stated properties. An instance of either side of $x^3+z\approx x^3$ contains $\varphi(x)^3$. Choosing an additive subterm $\ba$ of $\varphi(x)$ gives the additive subterm $\ba^3$, in which some variable occurs at least three times.

Similarly, an instance of either side of $x^2y^2+z\approx x^2y^2$ contains $\varphi(x)^2\varphi(y)^2$. Choose additive subterms $\ba$ and $\bb$ of $\varphi(x)$ and $\varphi(y)$. Its expansion contains $\ba^2\bb^2$. If $\ct{\ba}\cap\ct{\bb}\neq\varnothing$, some variable occurs at least four times. Otherwise two distinct variables each occur at least twice.

An additive context retains the indicated additive subterm; multiplication by a nonempty term retains a product having it as a factor. Thus the term before the first application would violate one of the two properties. No such first application exists, and the derivation uses only $\Sigma$.
\end{proof}

\begin{theorem}\label{NFB-interval}
Every variety $\cV$ with $\cD\leq\cV\leq\cM$ is nonfinitely based.
\end{theorem}
\begin{proof}
Suppose $\Delta$ is a finite basis for $\cV$. Since every member of $\Delta$ holds in $\cD$, compactness of equational logic provides a finite set $\Sigma\subseteq\Id(\bN)$ such that
\[
 \Sigma\cup\{x^3+z\approx x^3,\ x^2y^2+z\approx x^2y^2\}
 \vdash\Delta.
\]
Since $\cV\leq\cM$, the set $\Delta$ implies every $\varepsilon_n$. Lemma~\ref{no-first-use} then gives $\Sigma\vdash\varepsilon_n$ for every $n\geq3$, contrary to \eqref{AEI-independent}.
\end{proof}

\begin{proposition}\label{D-strict}
The inclusions $\cW_2<\cD<\cW_3$ are strict.
\end{proposition}
\begin{proof}
If all squares are greatest, then so are all cubes and all products of two squares. Hence $\cW_2\leq\cD\leq\cW_3$.

The subsemiring $\{1,2,3\}$ of $A_3$ has every product of at least three factors equal to $3$. Thus it belongs to $\cD$, but $1^2=2$ is not greatest, so it is outside $\cW_2$.

For the other inclusion, consider $T=\bN^2\setminus\{(0,0)\}$. Collapse the additive order filter and multiplicative ideal consisting of pairs having a coordinate at least $3$. The quotient has universe
\[
 \bigl(\{0,1,2\}^2\setminus\{(0,0)\}\bigr)\cup\{\infty\}.
\]
On the pairs, addition is coordinatewise maximum. Multiplication is coordinatewise numerical addition when both resulting coordinates are at most $2$, and is $\infty$ otherwise. The element $\infty$ absorbs both operations. This quotient belongs to $\cM$, every cube equals $\infty$, but
\[
 (1,0)^2(0,1)^2=(2,2)\neq\infty.
\]
It lies in $\cW_3\setminus\cD$.
\end{proof}

\section{Identities making additive subterms greatest}\label{greatest}
Let $\Omega\subseteq\Xc$ be arbitrary. For each $\bw\in\Omega$, choose a variable $z_{\bw}\notin\ct{\bw}$, and define
\begin{equation}\label{V-Omega}
 \cV_\Omega=\cM\cap[\,\bw+z_{\bw}\approx\bw:\bw\in\Omega\,].
\end{equation}
We first treat two types of a single additional identity.
\begin{lemma}\label{one-greatest}
The following varieties are finitely based:
\begin{align}
 &\cM\cap[x_1\cdots x_r+z\approx x_1\cdots x_r]
       &&(r\geq1),\label{linear-greatest}\\
 &\cM\cap[x^2y_1\cdots y_m+z\approx x^2y_1\cdots y_m]
       &&(m\geq0),\label{single-square-greatest}
\end{align}
where the displayed variables are distinct except for the repeated $x$.
\end{lemma}
\begin{proof}
Consider \eqref{single-square-greatest}, and put $d=m+1$. Its defining identity and \eqref{A} make every nonlinear additive subterm of length at least $d+1$ greatest. Indeed, choose two occurrences of one variable for $x^2$, choose $m$ further factors for the $y_i$, and multiply by any remaining factors.

Fix $d$ variables. Let $\Gamma_d$ consist of all identities of $\bN$ between the finitely many terms on these variables whose additive subterms have length at most $d$, modulo commutativity and additive idempotence. We prove that the commutative ai-semiring axioms, \eqref{A}, the defining identity in \eqref{single-square-greatest}, and $\Gamma_d$ form a basis.

Let $\bu\approx\bu+\bq$ hold in $\bN$. If $\bu$ has a nonlinear additive subterm of length at least $d+1$, it is greatest, and the identity follows. Otherwise every additive subterm of $\bu$ is linear or has length at most $d$.

Suppose $\len{\bq}>d$. Evaluate each variable of $\bq$ at numerical $1$ and all others at $0$. An additive subterm of length at most $d$ has value at most $d$, and a linear additive subterm has value at most $|\ct{\bq}|$. Since $\bq$ has value $\len{\bq}$, validity forces $\bq$ to be linear, and some linear additive subterm of $\bu$ must contain every variable of $\bq$. Then \eqref{factor} proves the identity.

Now suppose $\len{\bq}\leq d$. Delete the variables outside $\ct{\bq}$ from the additive subterms of $\bu$, omitting products that become empty, and call the resulting term $\bv$. By \eqref{content-N}, this term is nonempty. Evaluation of the deleted variables at numerical $0$ gives $\bN\models\bv\approx\bv+\bq$. Each additive subterm of $\bv$ has length at most $d$: an originally short subterm remains short, and an originally linear one has at most $|\ct{\bq}|\leq d$ variables. The identity therefore belongs to $\Gamma_d$ after renaming. Using \eqref{factor},
\[
 \bu\approx\bu+\bv\approx\bu+\bv+\bq\approx\bu+\bq.
\]
This proves finite basedness in \eqref{single-square-greatest}.

For \eqref{linear-greatest}, the case $r=1$ is the trivial variety. For $r\geq2$, put $d=r-1$. Every additive subterm of length at least $r$ is greatest. If none occurs in $\bu$, all its additive subterms have length at most $d$. Evaluation at $1$ on all variables gives $\len{\bq}\leq d$. The same argument with $\Gamma_d$ applies.
\end{proof}

\begin{theorem}\label{Omega-classification}
For $\cV_\Omega$ defined in \eqref{V-Omega}, the following conditions are equivalent:
\begin{enumerate}
\item $\cV_\Omega$ is finitely based;
\item some $\bw\in\Omega$ has every variable occurring at most twice and at most one variable occurring twice;
\item $\cD\nleq\cV_\Omega$.
\end{enumerate}
\end{theorem}
\begin{proof}
Suppose (ii) holds. If the chosen $\bw$ is linear of length $r$, its identity implies that of every additive subterm of length at least $r$. At lengths below $r$, only finitely many multiplicity patterns occur up to renaming variables. Retain one representative from each such pattern in $\Omega$. By Lemma~\ref{one-greatest}, the resulting finite additional set is adjoined to a finitely based variety, proving (i).

Otherwise the chosen additive subterm has the form $x^2y_1\cdots y_m$ after renaming. Its identity makes every nonlinear additive subterm of length at least $m+2$ greatest. Only finitely many nonlinear patterns remain at smaller lengths. If $\Omega$ also contains linear additive subterms, retain one of minimum length; its identity implies those of every longer linear additive subterm. Thus finitely many additional identities suffice, and Lemma~\ref{one-greatest} again proves (i).

If (ii) fails, every $\bw\in\Omega$ has a factor $x^3$, or a factor $x^2y^2$ with $x\neq y$. The defining identities of $\cD$ imply $\bw+z_{\bw}\approx\bw$ in either case. Hence $\cD\leq\cV_\Omega$, and Theorem~\ref{NFB-interval} makes $\cV_\Omega$ nonfinitely based. This proves (i)$\Rightarrow$(ii) and (iii)$\Rightarrow$(ii). Finally, (i)$\Rightarrow$(iii) follows from Theorem~\ref{NFB-interval}.
\end{proof}

\begin{corollary}\label{W-classification}
The variety $\cW_2$ is defined by the commutative ai-semiring axioms, \eqref{A}, and $x^2+z\approx x^2$. It is finitely based but not hereditarily finitely based. For every $k\geq3$, the variety $\cW_k$ is nonfinitely based.
\end{corollary}
\begin{proof}
Take $m=0$, and hence $d=1$, in the proof of Lemma~\ref{one-greatest}. The set $\Gamma_1$ contributes only trivial identities, giving the stated basis. Since $\TR\in\cW_2$ by~\cite[Lemma~3.1]{SRG}, the variety $\cW_2$ is not hereditarily finitely based. If $k\geq3$, then $\cD\leq\cW_3\leq\cW_k$, and Theorem~\ref{NFB-interval} applies.
\end{proof}

\begin{proposition}\label{W-chain}
We have
\[
 \cW_2<\cW_3<\cW_4<\cdots,\qquad
 \bigvee_{k\geq2}\cW_k=\cM.
\]
Every finite member of $\cW_k$ has nilpotent multiplicative reduct. There is no uniform bound on its nilpotency index, even among the finite members of $\cW_2$.
\end{proposition}
\begin{proof}
A greatest $k$th power is multiplicatively absorbing, giving $\cW_k\leq\cW_{k+1}$. The subsemiring $C_{k+1}=\{1,\ldots,k+1\}$ of $A_{k+1}$ belongs to $\cW_{k+1}$ but not to $\cW_k$, since $1^k=k<k+1$. The positive-integer subsemiring of $\bN$ embeds in $\prod_{k\geq2}C_k$ by truncation. It generates $\cM$ by Corollary~\ref{tails}, proving the join assertion.

Let $S\in\cW_k$ be finite, with greatest element $b$ and $|S|=m$. A product of $(k-1)(m-1)+1$ factors either contains $b$ or contains $k$ occurrences of the same other element. Commutativity then makes the product equal to $b$.

For unboundedness, let $n\geq1$ and start with $\bN^n\setminus\{\mathbf0\}$ and collapse the ideal and order filter of vectors with a coordinate at least $2$. The quotient has universe
\[
 Q_n=\bigl(\{0,1\}^n\setminus\{\mathbf0\}\bigr)\cup\{\infty\}.
\]
Every square is $\infty$, so $Q_n\in\cW_2$. The product of the $n$ coordinate vectors is $(1,\ldots,1)\neq\infty$, whereas every product of $n+1$ nonzero zero-one vectors has a coordinate at least $2$ before truncation. Thus the nilpotency index is exactly $n+1$.
\end{proof}

\section{The varieties defined by power stabilization}\label{powers}
For $k\geq1$, put
\begin{equation}\label{Pk}
 \cP_k=\cM\cap[x^k\approx x^{k+1}].
\end{equation}
By Theorem~\ref{proper-lf}, every proper subvariety of $\cM$ lies below some $\cP_k$.

\begin{lemma}\label{combine}
Assume \eqref{A}, \eqref{B}, and $x^2\approx x^3$. Let $\br,\bs\in\Xc$ have every variable occurring at most twice. Define a product $\bw$ by the following rules: a variable occurs twice in $\bw$ if it occurs twice in at least one of $\br,\bs$; otherwise, it occurs once in $\bw$ exactly when it occurs in both $\br$ and $\bs$. If this product is nonempty, then
\begin{equation}\label{combine-id}
 \br+\bs\approx\br+\bs+\bw.
\end{equation}
\end{lemma}
\begin{proof}
Delete from $\br$ and $\bs$ the variables whose multiplicity pairs are $(1,0)$ or $(0,1)$. Denote the products of the variables with multiplicity pairs $(2,0),(2,1),(1,2),(0,2)$ by $\ba,\bb,\mathbf c,\mathbf d$, respectively. Let $\bh$ collect the variables with pairs $(1,1),(2,2)$, with those common multiplicities. The retained products and the prescribed product are
\[
 \br_0=\bh\ba^2\bb^2\mathbf c,\qquad
 \bs_0=\bh\mathbf d^2\bb\mathbf c^2,\qquad
 \bw=\bh\ba^2\mathbf d^2\bb^2\mathbf c^2.
\]
An absent factor in these products is omitted. If $\ba^2\bb$ has no factors, then $\bw$ is a nonempty factor of $\bs$. If $\mathbf d^2\mathbf c$ has no factors, then $\bw$ is a nonempty factor of $\br$. In either case, \eqref{factor} proves \eqref{combine-id}.

Otherwise $\ba^2\bb$ and $\mathbf d^2\mathbf c$ are nonempty substitution terms for \eqref{B}. By $x^2\approx x^3$,
\[
\begin{aligned}
 \br_0+\bs_0
 &\approx\bh\ba^4\bb^3\mathbf c+\bh\mathbf d^4\bb\mathbf c^3\\
 &=\bh\bb\mathbf c\bigl((\ba^2\bb)^2+(\mathbf d^2\mathbf c)^2\bigr)\\
 &\approx\bh\bb\mathbf c\bigl((\ba^2\bb)^2+(\mathbf d^2\mathbf c)^2
                          +(\ba^2\bb)(\mathbf d^2\mathbf c)\bigr)\just{\eqref{B}}\\
 &\approx\br_0+\bs_0+\bh\ba^2\mathbf d^2\bb^2\mathbf c^2.
\end{aligned}
\]
When $\bh\bb\mathbf c$ has no factors, the displayed application of \eqref{B} is made without an outer multiplier. Both $\br_0$ and $\bs_0$ are now nonempty factors of $\br$ and $\bs$, respectively. Thus \eqref{factor} gives
\[
 \br+\bs\approx\br+\bs+\br_0+\bs_0
 \approx\br+\bs+\br_0+\bs_0+\bw
 \approx\br+\bs+\bw.
\]
\end{proof}

\begin{theorem}\label{P2-basis}
The variety $\cP_2$ is defined by the commutative ai-semiring axioms, \eqref{A}, \eqref{B}, and $x^2\approx x^3$.
\end{theorem}
\begin{proof}
Let $\mathcal U$ be defined by the stated finite set. Then $\cP_2\leq\mathcal U$. We prove that every identity $\bu\approx\bu+\bq$ of $\bN$ holds in $\mathcal U$.

Apply \eqref{convex} to the original additive subterms of $\bu$ and to $\bq$. After omitting coefficients equal to zero, there is a nonempty index set $I$ such that
\begin{equation}\label{positive-combination}
 \alpha(\bq)\leq\sum_{i\in I}\lambda_i\alpha(\bu_i),
 \qquad\lambda_i>0,\qquad\sum_{i\in I}\lambda_i=1.
\end{equation}
Let $\overline{\bu}_i$ and $\overline{\bq}$ be obtained by replacing every exponent greater than two by two. These replacements are valid in $\mathcal U$ by $x^2\approx x^3$.

If every selected $\bu_i$ is linear, \eqref{positive-combination} forces $\bq$ to be linear. Each variable of $\bq$ occurs in every selected $\bu_i$: a positive weighted average of numbers in $\{0,1\}$ is at least one only if all the numbers equal one. Thus $\bq$ is a factor of each selected $\bu_i$, and \eqref{factor} applies.

Otherwise begin with a selected $\overline{\bu}_i$ containing a square, and apply Lemma~\ref{combine} successively with the other selected $\overline{\bu}_j$. At every step that square is retained, so the prescribed product is nonempty. The resulting additive subterm $\bw$ satisfies
\[
 \sum_{i\in I}\overline{\bu}_i
 \approx\sum_{i\in I}\overline{\bu}_i+\bw.
\]
A variable occurs twice in $\bw$ exactly when it occurs at least twice in a selected original $\bu_i$. Every remaining variable occurs once in $\bw$ exactly when it occurs in all the selected $\bu_i$.

If a variable occurs at least twice in $\bq$, equation~\eqref{positive-combination} implies that it occurs at least twice in some selected $\bu_i$. If it occurs once in $\bq$, then either some selected $\bu_i$ contains it at least twice, or its selected exponents all belong to $\{0,1\}$. In the latter case they must all equal one. It follows that $\overline{\bq}$ is a factor of $\bw$. Adding the unselected additive subterms and using \eqref{factor} proves $\bu\approx\bu+\bq$ in $\mathcal U$.

Therefore $\mathcal U\leq\cM$ by \eqref{one-sided}. Its defining power identity then gives $\mathcal U=\cP_2$.
\end{proof}

\begin{theorem}\label{P-classification}
For $k\geq1$, the variety $\cP_k$ is finitely based if and only if $k\leq2$. Moreover,
\[
 \cP_1<\cP_2<\cP_3<\cdots,\qquad
 \bigvee_{k\geq1}\cP_k=\cM.
\]
\end{theorem}
\begin{proof}
Under $x^2\approx x$, distributivity and \eqref{A} give
\[
 x+y\approx(x+y)^2\approx x+y+xy\approx xy.
\]
Thus $\cP_1$ is the variety with both operations equal to a semilattice operation, and is finitely based. The case $k=2$ is Theorem~\ref{P2-basis}. For $k\geq3$, each member of $\cD$ has all cubes greatest and multiplicatively absorbing. Hence $\cD\leq\cP_k$, and Theorem~\ref{NFB-interval} gives nonfinite basedness.

Multiplying $x^k\approx x^{k+1}$ gives $\cP_k\leq\cP_{k+1}$. The algebra $A_{k+1}$ belongs to $\cP_{k+1}$ but not to $\cP_k$, so each inclusion is strict. Since $A_k\in\cP_k$, equation~\eqref{Ak-chain} proves the join assertion.
\end{proof}
\begin{corollary}\label{no-maximal}
The variety $\cM$ has no maximal proper subvariety.
\end{corollary}
\begin{proof}
For any $\cV<\cM$, Theorem~\ref{proper-lf} gives $\cV\leq\cP_k$ for some $k$. Then $\cV\leq\cP_k<\cP_{k+1}<\cM$ by Theorem~\ref{P-classification}.
\end{proof}

\section{The atoms and a restriction map}\label{atoms}
Let $M_2$ have universe $\{0,1\}$ with both operations equal to maximum. Let $T_2$ have the same addition and constant multiplication $ab=1$. These are the standard two-element ai-semirings of~\cite{YSR}; the latter is also used in~\cite{SRG}. Both belong to $\cM$: $M_2\cong A_1$ and $T_2$ is isomorphic to the subsemiring $\{1,2\}$ of $A_2$. We denote the trivial variety by $\cT$. We also use the standard basis for $\V(T_2)$: relative to the commutative ai-semiring axioms, it is defined by \eqref{A} and $xy\approx zt$~\cite{YSR}.

\begin{proposition}\label{atoms-exclude}
For every $\cV\leq\cM$,
\begin{align}
 T_2\notin\cV
 &\quad\Longleftrightarrow\quad\cV\models x^2\approx x
 \quad\Longleftrightarrow\quad\cV\leq\V(M_2),\label{T2-exclude}\\
 M_2\notin\cV
 &\quad\Longleftrightarrow\quad
 \cV\models x^k+z\approx x^k\text{ for some }k\geq1.\label{M2-exclude}
\end{align}
The only atoms of the subvariety lattice of $\cM$ are $\V(M_2)$ and $\V(T_2)$.
\end{proposition}
\begin{proof}
Suppose $T_2\notin\cV$ and choose an identity $\bu\approx\bu+\bq$ of $\cV$ failing in $T_2$. Since every product of length at least two is greatest in $T_2$, all additive subterms of $\bu$ have length one. If $\len{\bq}\geq2$, identify all variables to get $x\approx x^{\len{\bq}}$. Equation~\eqref{A} then gives $x\approx x^2$. If $\bq$ has length one, its variable is absent from $\bu$; map it to $x^2$ and all variables of $\bu$ to $x$, with the same conclusion. Under $x^2\approx x$, the calculation in the proof of Theorem~\ref{P-classification} gives $xy\approx x+y$. Both operations therefore coincide with a semilattice operation. Conversely, $x^2\approx x$ fails in $T_2$, proving \eqref{T2-exclude}.

The identities of $M_2$ are exactly those with the same content on both sides. If $M_2\notin\cV$, choose an identity $\bu\approx\bu+\bq$ of $\cV$ and $t\in\ct{\bq}\setminus\ct{\bu}$. Map $t$ to $z$ and every other variable to $x$. If $k$ is the maximum length of an additive subterm of $\bu$, its image reduces to $x^k$ by \eqref{A}. The image of $\bq$ has $z$ as a factor, giving
\[
 x^k\approx x^k+\varphi(\bq)
 \approx x^k+\varphi(\bq)+z
 \approx x^k+z.
\]
No such identity holds in $M_2$, which proves \eqref{M2-exclude}.

The varieties $\V(M_2)$ and $\V(T_2)$ are minimal nontrivial varieties by~\cite{YSR}. If a subvariety of $\cM$ omitted both, \eqref{T2-exclude} and \eqref{M2-exclude} would imply $x+z\approx x$, making it trivial. Thus there are no other atoms.
\end{proof}

Put $\cW=\V(M_2,T_2)$.
\begin{proposition}\label{join-atoms}
Relative to the commutative ai-semiring axioms, $\cW$ is defined by \eqref{A} and
\begin{align}
 x^2&\approx x^3,\label{join1}\\
 xy&\approx x^2+y^2,\label{join2}\\
 x^2+y&\approx x^2+y^2.\label{join3}
\end{align}
Its subvariety lattice consists of $\cT$, $\V(M_2)$, $\V(T_2)$ and $\cW$, and $\cW<\V(A_2)$.
\end{proposition}
\begin{proof}
Both generators satisfy the displayed identities. Equations~\eqref{join1} and \eqref{join2} express every additive subterm of length at least two as the sum of the squares of its content variables. For the induction step, use
\[
 x^2y\approx x^4+y^2\approx x^2+y^2.
\]
If a term has an additive subterm of length at least two, \eqref{join3} also replaces every remaining additive subterm of length one by its square. Every term thus has one of the two forms
\[
 \sum_{x\in C}x,\qquad\sum_{x\in C}x^2,
\]
according as all original additive subterms have length one or at least one has greater length. In $M_2$, equality requires equal contents. In $T_2$, the second form is constantly greatest, whereas the first is not. These observations distinguish any two different represented terms and prove that the displayed identities define $\cW$.

A proper subvariety cannot contain both generators. If it omits $T_2$, equation~\eqref{T2-exclude} places it below $\V(M_2)$. If it omits $M_2$, equation~\eqref{M2-exclude} gives $x^k+z\approx x^k$. For $k=1$ the variety is trivial. For $k\geq2$, \eqref{join1} gives $x^2+z\approx x^2$, so all squares coincide, and \eqref{join2} makes multiplication constant. The standard basis for $T_2$ then places the variety below $\V(T_2)$~\cite{YSR}. This proves the lattice assertion.

The subsets $\{0,2\}$ and $\{1,2\}$ of $A_2$ are isomorphic to $M_2$ and $T_2$, respectively. Hence $\cW\leq\V(A_2)$. Identity~\eqref{join2} fails in $A_2$ at $x=0$, $y=1$, proving strictness.
\end{proof}

\begin{corollary}\label{fibres}
Let $\mathcal L(\mathcal U)$ denote the subvariety lattice of $\mathcal U$, and let
\[
 \pi:\mathcal L(\cM)\longrightarrow\mathcal L(\cW),\qquad
 \pi(\cV)=\cV\cap\cW.
\]
Then
\begin{align*}
 \pi^{-1}(\cT)&=\{\cT\},&
 \pi^{-1}(\V(M_2))&=\{\V(M_2)\},\\
 \pi^{-1}(\V(T_2))&=\bigcup_{k\geq2}[\V(T_2),\cW_k],&
 \pi^{-1}(\cW)&=[\cW,\cM].
\end{align*}
\end{corollary}
\begin{proof}
If $\cV\cap\cW=\cT$, then $\cV$ omits both atoms and is trivial. If the intersection is $\V(M_2)$, then $T_2\notin\cV$, so \eqref{T2-exclude} gives $\cV=\V(M_2)$. If the intersection is $\V(T_2)$, then $\cV$ contains $T_2$ and omits $M_2$. By \eqref{M2-exclude}, it lies below some $\cW_k$, necessarily with $k\geq2$. Conversely, each variety in such an interval contains $T_2$ and omits $M_2$. The last equality is immediate from $\pi(\cV)=\cW$ if and only if $\cW\leq\cV$.
\end{proof}

\section{A descending chain and the remaining problem}\label{last}
We conclude with a consequence of the infinite basis in~\cite[Proposition~3.5]{SRG}. For $n\geq1$, let
\[
 \bu^{(n)}=x_1x_2+x_2x_3+\cdots+x_{2n}x_{2n+1}+x_{2n+1}x_1,
\]
and let $\sigma_n$ be the identity $\bu^{(n)}+z\approx\bu^{(n)}$, where $z$ is a new variable. For $N\geq1$, let $\mathcal B_N$ be defined by the commutative ai-semiring axioms, \eqref{A}, and
\[
 x^2+z\approx x^2,\qquad xyz\approx x^2,\qquad
 \sigma_1,\ldots,\sigma_N.
\]
\begin{proposition}\label{descending}
Each $\mathcal B_N$ is finitely based, and
\[
 \V(\TR)<\cdots<\mathcal B_{N+1}<\mathcal B_N\leq\V(A_2),
 \qquad\bigcap_{N\geq1}\mathcal B_N=\V(\TR).
\]
For $\bw\in\Xc$ and $z\notin\ct{\bw}$, the identity $\bw+z\approx\bw$ holds in some $\mathcal B_N$ if and only if it holds in all $\mathcal B_N$, if and only if it holds in $\TR$. These conditions are equivalent to $\len{\bw}\geq3$ or $\bw$ being nonlinear.
\end{proposition}
\begin{proof}
Finite basedness follows from the finite defining set. The basis in~\cite[Proposition~3.5]{SRG} gives the intersection equality. The proof of~\cite[Proposition~3.6]{SRG} shows that $\mathcal B_N$ fails $\sigma_{N+1}$, so $\mathcal B_{N+1}<\mathcal B_N$. Also $\V(\TR)<\mathcal B_N$ because the former is nonfinitely based.

In $\mathcal B_N$, every square is greatest, so \eqref{A2-basis1} holds. Identity $\sigma_1$ makes $xy+yz+xz$ greatest, and $xyz\approx x^2$ does the same for triple products. Thus \eqref{A2-basis2} holds, and the cited basis for $A_2$ gives $\mathcal B_N\leq\V(A_2)$.

A nonlinear additive subterm has a square factor, and every product of length at least three is greatest in $\mathcal B_N$. Conversely, a remaining $\bw$ is a variable or a product of two distinct variables. In $\TR$, assign $5$ to the former, or $5,6$ to the factors of the latter, and assign $1$ to $z$. The value of $\bw$ is $5$ or $3$, respectively, and $\bw+z\approx\bw$ fails. Since $\TR\in\mathcal B_N$ for every $N$, it fails there as well.
\end{proof}

Thus the identities making individual additive subterms greatest do not determine finite basedness among arbitrary subvarieties of $\cM$. Theorem~\ref{main} determines hereditary finite basedness, and Theorems~\ref{NFB-interval}, \ref{Omega-classification} and~\ref{P-classification} determine ordinary finite basedness for the families considered above. The remaining problem can be stated as follows.
\begin{problem}\label{remaining}
Determine which subvarieties $\cV$ satisfying
\[
 \V(\TR)\leq\cV\leq\cM,\qquad \cD\nleq\cV,
\]
are finitely based. In particular, determine the finitely based subvarieties in $[\V(\TR),\V(A_2)]$.
\end{problem}
Every variety in this problem is a proper subvariety of $\cM$ and is locally finite by Theorem~\ref{proper-lf}. Proposition~\ref{descending} provides a descending sequence of finitely based members of this interval with nonfinitely based intersection.

\end{document}